[1994/12/01]
\documentclass[reqno,12pt]{amsart}
\usepackage{bbm}
\def\1{\mathbbm{1}}
\usepackage{amsmath,amssymb}
\usepackage{amsmath,amsthm}
\usepackage[mathscr]{euscript}
\usepackage{mathrsfs}
\usepackage{chngcntr}
\usepackage{amsfonts}
\usepackage{amssymb}
\usepackage{latexsym}
\usepackage{tikz}
\usepackage{graphics}

\newtheorem{definition}{Definition}[section]
\newtheorem{remark}[definition]{Remark}
\newtheorem{proposition}[definition]{Proposition}
\newtheorem{theorem}[definition]{Theorem}
\newtheorem{lemma}[definition]{Lemma}

\newtheorem{example}[definition]{Example}

\numberwithin{equation}{section}

\newcommand{\T}{{\mathbb T}}                   

\newcommand{\B}{{\mathbb B}}

\def\al{\alpha}
\def\om{\omega}
\def\Om{\Omega}
\def\ga{\gamma}

\def\la{\lambda}

\def\La{\Lambda}

\DeclareMathOperator{\tr}{tr}

\def\calL{{\mathcal{L}}}

\def\calB{{\mathcal{B}}}

\def\calC{{\mathcal{C}}}

\def\calF{{\mathcal{F}}}

\def\P{\mathbb P}

\def\E{\mathbb E}

\def\T{\mathbb T}

\title{Solving stochastic equations with unbounded nonlinear perturbations}
\author{M. Fkirine and S. Hadd}
\address{Department of Mathematics, Faculty of Sciences, Ibn Zohr University, Hay Dakhla, B.P. 8106, 80000 Agadir, Morocco; fkirinemohamed@gmail.com, s.hadd@uiz.ac.ma}
\thanks{The first author was supported by National Center for Scientific and Technical Research
  	(CNRST), Morocco.}

\begin{document}
\maketitle

\renewcommand{\sectionmark}[1]{}

\begin{abstract}
This paper is interested in semilinear stochastic equations having unbounded nonlinear perturbations in the deterministic part and/or in the random part. Moreover the linear part of these equations is governed by a not necessarily analytic semigroup. The main difficulty with these equations is how to define the concept of mild solutions due to the chosen type of unbounded perturbations. To overcome this problem, we first proved a regularity property of the stochastic convolution with respect to the domain of ``admissible" unbounded linear operators (not necessarily closed or closable). This is done using Yosida extensions of such unbounded linear operators. After proving the well-posedness of these equations, we also establish the Feller property for the corresponding transition semigroups. Several examples like heat equations and schr\"odinger equations with nonlocal perturbations terms are given. Finally, we give an application to a general class of semilinear neutral stochastic equations.
\end{abstract}

\maketitle

\section{Introduction}
We are concerned with the well-posedness and Feller property of two classes of abstract stochastic equations under unbounded nonlinear perturbations. We will work with generators that are not necessarily analytic, so we do not assume any maximal regularity property for the linear part of the stochastic equation. More precisely, we consider the following abstract stochastic equations
\begin{align}\label{Flandoli-exa}
dX(t)=AX(t)dt+\mathscr{M}(\mathscr{B} X(t))dW(t),\quad X(0)=\xi,\quad t\in [0,T],
\end{align}
and

 \begin{eqnarray}\label{0}
\begin{cases}
dX(t)=[AX(t)+F(\mathscr{C}X(t))]dt+\mathscr{M}(X(t))dW(t), & t\in[0,T]\\
X(0)=\xi
\end{cases}
\end{eqnarray}
in a real separable Hilbert space $H$. Here, $A:D(A)\subset H\to H$ is the generator of a strongly continuous semigroup $\T:=(T(t))_{t\geq0}$ on $H$ (not necessarily analytic), $\mathscr{C},\mathscr{B}: \mathscr{Z}\subset H\to H$ are linear operators (not necessarily closed or closable nor represented by fractional operators), $D(A)\subset \mathscr{Z}\subset H,$ $F:H\to H$ and $\mathscr{M}:H\to L_2^0(V,H)$ are nonlinear mappings, where $L_2^0(V,H)$ is the space of all Hilbert-Schmidt operators acting between $Q^\frac{1}{2}(V)$ and $H$, see notation below. Moreover, $W(t)$ is a $Q$-Wiener process in $V$ defined on a filtered probability space $(\Om,\calF,(\calF_t)_{t\ge 0},\P)$. The initial condition $\xi$ can be any $\mathcal{F}_0$-measurable $H$-valued random variable.

By selecting \begin{align}\label{G}G=F\circ \mathscr{C}:\mathscr{Z}\to H,\end{align} the equation \eqref{0} becomes
\begin{eqnarray}\label{****}
\begin{cases}
dX(t)=[AX(t)+G(X(t))]dt+\mathscr{M}(X(t))dW(t), & t\in[0,T]\\
X(0)=\xi.
\end{cases}
\end{eqnarray}
In the particular case of $\mathscr{C}\in\calL(H)$, that is $D(G)=H$, the equation \eqref{****} has been extensively studied. We cannot give a complete description of the literature, however we cite the following references \cite{da2014stochastic}, \cite{Daprato1992},  \cite{ichikawa1982stability} \cite[Chap.4]{Pr-Ro},  \cite{Peszat1995}, and \cite{Michalik1995} in which  the existence and uniqueness of solution as well as the regularity of solutions are established. We also mention that the case when the nonlinear application has a small domain $D(G)=\mathscr{Z}\subset H$ and an image in $\mathscr{Z}$ (this is $G:\mathscr{Z}\to \mathscr{Z}$) is also considered by many authors, see e.g. \cite{gatarek1997}, \cite{maslowski1993},  \cite{manthey1992}, \cite{maslowski1989} and \cite{da1988note}.  The common strategy used in the references mentioned above is to consider the following assumptions: the part of $A$ in $\mathscr{Z}$ generates a strongly continuous semigroup on $\mathscr{Z},$ $G$ is a locally Lipschitz or a continuous function from $\mathscr{Z}$ to $\mathscr{Z},$ and the initial condition $\xi$ is a $\mathscr{Z}$-valued random. As shown in  \cite[Theorem 7.15]{da2014stochastic}, to work with initial conditions $\xi\in L^2(\mathcal{F}_0, H)$, extra assumptions on $G:\mathscr{Z}\to \mathscr{Z}$ are required, and in this case we only talk about a special class of solutions called ``generalized" solutions.  Also the equation \eqref{Flandoli-exa} is studied in \cite[page 176]{da2014stochastic} and \cite{Flandoli1} where the generator $A$ satisfies some maximal regularity properties.

In the present paper, we will use the following assumptions for the semilinear equation \eqref{****}:
\begin{itemize}
  \item [{\bf(S)}] The operator $A:D(A)\subset H\to H$ generates a strongly continuous semigroup $\T:=(T(t))_{t\ge 0}$ on $H,$ such that
  \begin{align*}
\|T(t)\|\le M e^{\beta t},\qquad \forall t\ge 0
\end{align*}
for some $\beta>\om_0(A)$ and $M\ge 1$.
  \item [{\bf(G)}] The nonlinear map $G:\mathscr{Z}\to H$ is given by \eqref{G}.
  \item [{\bf(L)}] $F:H\rightarrow H$, $\mathscr{M}:H\rightarrow L_2^0(V,H)$ are nonlinear, Lipschitz continuous operators, i.e,
$$\|F(x)-F(y)\|+\|\mathscr{M}(x)-\mathscr{M}(y)\|_2\leq k\|x-y\|,~~\text{for all } x,y\in H.$$
\item [{\bf(A)}] The operator $C:=\mathscr{C}$ with domain $D(C):=D(A)$ is an admissible operator for $A$. That is, for some (hence all) $\alpha>0$ there exists $\ga:=\gamma(\alpha)>0$ such that
\begin{equation}\label{adm}
\int_0^\alpha \|CT(t)x\|^2dt\leq \gamma^2 \|x\|^2,\qquad (x\in D(A)).
\end{equation}
\item[{\bf(A')}] $B:=\mathscr{B}$ with domain $D(B):=D(A)$ is a zero-class admissible operator for $A$. That is, $B$ satisfies the condition \eqref{adm} with a constant $\tilde{\ga}(\al)\to 0$ as $\al\to 0$.
\end{itemize}

We mention that in the condition {\bf (G)}, we may have the following situation  $\mathscr{Z}\subsetneqq{\rm Im}(G)\subseteq H$. This case goes beyond the approach used in the references cited above. Also, unlike these references,  we do not assume that the space $\mathscr{Z}$ is stable by the semigroup $\T$, so we can not use the part of $A$ in $\mathscr{Z}$ to define a semigroup on $\mathscr{Z}$. This fact may implies that the determinist and stochastic convolutions may  exceed $\mathscr{Z}$. This offers many difficulties in defining naturally a mild solution for the problem \eqref{****} (or \eqref{0}). To overcome this problem, we will use the fact that $\mathscr{C}$ satisfies the condition {\bf(A)} and the concept of Yosida extensions (see \eqref{Yosida-extension}) to give sense to the eventual solution to the semilinear problem \eqref{0}. In fact, we first show in Theorem \ref{6*} that under the conditions {\bf(A)} and {\bf(S)}, we have
\begin{align*}
W_A^\Phi(t):=\int_0^t T(t-s)\Phi(s)dW(s)\in D(C_\Lambda),\qquad a.e.\;t\ge 0,\quad \mathbb{P}-a.s,
\end{align*}
where $C_\Lambda$ is the Yosida extension of $C$ for $A$. Furthermore, we have the regularity estimate \eqref{Reg-Max}. We note that in the case of analytic semigroup $\T$ and $\mathscr{C}=(-A)^\al$ for some $\al\in (0,\frac{1}{2})$, the above result is proved if \cite[Theorem 6.14]{da2014stochastic}, where the proof is heavily based on properties of analytic semigroups. This case can also considered as a corollary of our result Theorem \ref{6*}, because the analyticity of $\T$ implies that $(-A)^\al$ satisfies the condition {\bf(A)}.

In Theorem \ref{B-theorem}, we assumed that the operator $B$ satisfies the condition {\bf(A')} and proved that the stochastic equation \eqref{Flandoli-exa} has a unique mild solution. The proof is mainly based on Theorem \ref{6*} and a Banach's fixed point theorem. In fact, compared with \cite{bouna}, \cite[Theorem 6.21]{da2014stochastic} and \cite{Flandoli1}, we dot not assume any regularity of the semigroup $\T$. We note that the condition  {\bf(A')} is somehow similar to a condition introduced in  \cite{Flandoli1}.

By using the above result on the stochastic convolution, the assumptions {\bf(A)}, {\bf(L)}, {\bf(G)}, and {\bf(S)},  and  an approximation method, we prove (see Theorem \ref{thm1}) the existence and uniqueness of a process $X\in \mathcal{C}_\mathcal{F}([0,T],H)$ such that $X(t)\in D(C_\La)$ a.e. $t>0,$ $\mathbb{P}$-a.s., and  satisfying
\begin{align*}
X(t)=T(t)\xi+\int^t_0 T(t-s)F(C_\Lambda X(s))ds+\int^t_0 T(t-s)\mathscr{M}(X(s))dW(s)
\end{align*}
for any $t\ge 0$ and $\xi\in L^2(\mathcal{F}_0,H)$. The particular case when $F\in\calL(H)$ is recently proved in \cite[Theorem 3.6]{lahbiri2020}. On the other hand, in the deterministic case, the authors of \cite{thieme2004semilinear} considered the equation \eqref{0} with $\mathscr{M}\equiv 0$ and assumed that $\mathscr{C}$ satisfies {\bf (A)}. They proved the existence and uniqueness of a continuous mild solution, but the above variation constant formula is not established. We also mention that the proof of Theorem \ref{thm1} can also be used for the deterministic case which we consider as an alternative proof to the corresponding proof in \cite{thieme2004semilinear}. Another result in the deterministic case can be found in \cite[Section 11]{curtain2020introduction} and \cite[Chap.7]{lunardi2012analytic}, in which a very particular case is considered.

In Section \ref{sec:feller}, we consider the transition semigroup associated with the equation \eqref{0},
\begin{align*}
P_t \phi=\E\left(\phi(X(t,x))\right)
\end{align*}
for any $t\ge 0,\,x\in H$ and any $\phi\in \calC_b(H),$ the space of all bounded and Borel measurable functions $\phi:H\to \mathbb{R}$. In Proposition \ref{prop2}, a continuous dependence of the solutions on initial conditions is established, which implies the Markov and Feller properties, in particular $P_t \calC_b(H)\subset \calC_b(H)$.

We end this paper with an application to semilinear stochastic equations of neutral type. We will use product spaces to introduce a semigroup approach to the well-posedness of such equations.

In the literature, one can find the typical semilinear stochastic equation
\begin{align}\label{KPZ}
\begin{cases}
dX(t,x)=\left[\partial_x^2X(t,x)+\frac{1}{2}\left(\partial_x X(t,x)\right)^2\right]dt+dW(t,x), & t\in[0,T],\quad x\in\mathbb{R}\\
X(0,x)=\xi(x)\in L^2(\mathbb{R}),
\end{cases}
\end{align}
where $W(t)$ is a $Q$-Wiener process on a Hilbert space $L^2(\mathbb{R})$. In fact, equation \eqref{KPZ} is called the Kardar-Parisi-Zhang equation (KPZ equation) and was introduced in \cite{KPZ1986} as a model of the interface growing in the phase transitions theory. See also \cite{BB-16}, \cite{GL-20}, \cite{Hair-13}, \cite{HQ-18} for more details about such equations. The singular term $(\partial_x X(t,x))^2$ makes the study of the well-posed of the equation \eqref{KPZ} very difficult. In fact, this question was an open problem for over 15 years until Hairer's key paper \cite{Hair-13} appears.
It is known that the equation \eqref{KPZ} can be reformulated as a Burgers equation,
which is investigated in \cite{daprato2007}, see also \cite[Chap.13]{da2014stochastic} and \cite[Chap.14]{Da-Za-Book-96}.

If we take $A=\partial_x^2,$ $f(x)=x^2$, $\mathscr{C}=\partial_x$, and $F(\phi)(x)=f(\phi(x)),$ then the KPZ equation \eqref{KPZ} has the same form as our abstract equation \eqref{0}. The only difference is that in the equation \eqref{0} the function $F$ is locally Lipschitz function. In our case, we only treat the case when $F$ is  global Lipschitz function (see the condition {\bf(L)}).

The organization of this paper is as follows: In Section \ref{sec:2} we first prove a technical result on stochastic convolutions and then use this result to prove the well-posedness of the stochastic equation \eqref{Flandoli-exa}, while in Section \ref{sec:well-posed} we prove the well-posedness of the equation \eqref{0}. In section \ref{sec:feller}, we establish the Feller property of the transition semigroup associated with the solution of the equation \eqref{0}. The last section is devoted to the well-posedness of a large class of neutral stochastic equations.

\noindent {\bf Notation.} Throughout this paper, $H$ denotes a real separable Hilbert space with norm $\|\cdot\|$ and $(A,D(A))$ is the generator of a $C_0$-semigroup $(T(t))_{t\geq0}$ on $H$ with growth bound $\omega_0$.  For $\lambda\in \rho(A)$ (the resolvent set of $A$) we set $R(\lambda,A):=(\lambda-A)^{-1}$. The Yosida extension of a linear operator $C:D(A)\to H$ for $A$ is the following linear operator
\begin{align}\label{Yosida-extension}
\begin{split}
D(C_{\Lambda})&:=\left\{x\in H: ~~\lim_{\lambda\to+\infty}C\lambda R(\lambda,A)x \quad\text{exists in}\quad H\right\},\\
C_{\Lambda}x&:= \lim_{\lambda\in\mathbb{R},\,\lambda\to+\infty}C\lambda R(\lambda,A)x.
\end{split}
\end{align}
Clearly, we have $D(A)\subset D(C_\Lambda)\subset H$ and $C_\Lambda x= C x$ for any $x\in D(A)$.

Let $(\Omega,\mathcal{F},\mathbb{P})$ be a complete probability space equipped with a normal filtration $\{\mathcal{F}_t\}_{t\geq0}$. Let $(\beta_k(t))_k $ be a sequence of real-valued one-dimensional standard Brownian motions mutually independent over $(\Omega,\mathcal{F},\mathbb{P})$. Set
$$W(t)=\sum_{k=0}^\infty\sqrt{\lambda_k}\beta_k(t)e_k, \qquad t\geq0,$$
where $(\lambda_k)_{k>0}$ are nonnegative real numbers and $(e_k)_{k>0}$ is a complete orthonormal basis in $V$. Let $Q\in\mathcal{L}(V)$ be an operator defined by $Qe_k=\lambda_ke_k$ with finite trace $\tr Q=\sum_{k=1}^\infty \lambda_k<\infty$. Then the above $V$-valued stochastic process $W(t)$ is called a $Q$-Wiener process. Denote by $L_2^0(V,H)$ the Hilbert space of all Hilbert-Schmidt operators from $V$ to $H$ equipped with the norm:

\begin{equation*}
\|D\|_2:=\|D\|_{L_2^0(V,H)}=\left(\sum_{k=1}^{+\infty}\|\sqrt{\lambda_k}De_k\|^2\right)^{\frac{1}{2}}, \qquad D\in L_2^0(V,H).
\end{equation*}
Let $L^2(\mathcal{F}_T,H):=L^2(\Omega,\mathcal{F}_T,H)$ the Hilbert space of all $\mathcal{F}_T$-measurable square integrable. $L^2_{\mathcal{F}}([0,T];U)$ is the space of the $\mathcal{F}_t$-adapted, $U$-valued measurable process $u(t,\omega)$ on $[0,T]$ such that $\mathbb{E}\int_0^T\|u(t,\omega)\|^2dt<+\infty$. $\mathcal{C}_\mathcal{F}([0,T];U)$  is the space of the $\mathcal{F}_t$-adapted, $U$-valued measurable process $u(t,\omega)$ on $[0,T]$ such that $\left(\mathbb{E}\|u(t,\omega)\|^2\right)^\frac{1}{2}$ is continuous. For brevity, we suppress the dependence of all mappings on $\omega$ throughout the manuscript.

\section{Regularity of the stochastic convolutions for admissible operators}\label{sec:2}
We start this section by giving some known consequences of the condition {\bf(A)} introduced in the first section. The constant $\ga>0$ in \eqref{adm} is a non-decreasing function on $\al$ and  $\lim_{\alpha\to0}\alpha^r\gamma(\alpha)=0$ for any $r>0$. If $C\in\calL(D(A),H)$ satisfies the condition {\bf(A)} then by H\"older's inequality one can see that there exists $\al_0>0$ and a constant $\tilde{\ga}\in (0,1)$ such that $\|CT(\cdot)x\|_{L^1([0,\al_0],H)}\le \tilde{\ga}\|x\|$ for any $x\in D(A)$. Thus $C$ is a Miyadera-Voigt perturbation for $A$. Hence the operator $A^C:=A+C$ with domain $D(A^C):=D(A)$ generates a strongly continuous semigroup on $H,$ see \cite[p.199]{engel2001one}.

As shown in \cite{weiss1989admissible}, the condition {\bf (A)} implies that
\begin{align}\label{C-Weiss}
\begin{split}
&T(t)x\in D(C_\Lambda)\quad a.e.\; t>0,\cr &
\int_0^\alpha \|C_\Lambda T(t)x\|^2dt\leq \gamma(\alpha)^2 \|x\|^2,\quad \forall x\in H,\; \forall \al>0,
\end{split}
\end{align}
where $C_\Lambda$ is the Yosida extension of $C$ for $A$, see notation in the first section. The following result taken from \cite[prop.3.3]{hadd2005unbounded} shows a regularity property for the deterministic convolution.
\begin{proposition}\label{5*}
Assume that $A$ and $C\in\calL(D(A),H)$ satisfy the conditions {\rm\bf(S)} and {\rm\bf(A)}, respectively. Then
\begin{align*}&(\T*f)(t):=\int_0^t T(t-s)f(s)ds\in D(C_{\Lambda}),\quad a.e.\; t\geq0,\cr & \left\|C_{\Lambda}(\T*f)\right\|_{L^2([0,\alpha],H)}\leq c(\alpha)\|f\|_{L^2([0,\alpha],H)},\end{align*}
for all $f\in L^2_{loc}(\mathbb{R}^+,H)$ and $\alpha>0$ with $c(\alpha):=\alpha^\frac{1}{2}\gamma(\alpha)>0$ is independent of $f$. Moreover, $\lim_{\alpha\to0}c(\alpha)=0$.
\end{proposition}
In the following we will prove a somehow an analogue regularity for the stochastic convolution. Before doing this, let gather some facts about fractional powers of operators.
\begin{remark}\label{zwart1}
Assume that $A$ generates an analytic semigroup $\T:=(T(t))_{t\ge 0}$ on a Hilbert space $H$ and let $H_\beta:=D((-A)^\beta)$ for some $\beta\in (0,\frac{1}{2})$. This space is endowed with the norm $\|x\|_\beta=\|(-A)^\beta x\|$. We recall that $\|(-A)^\beta T(t)\|\le \frac{M}{t^\beta}$ for $t>0,$ due to the analyticity of $\T$, see e.g. \cite{engel2001one}. Now clearly $K:=(-A)^\beta|_{D(A)}$ satisfies the condition {\bf(A)}, and its  Yosida extension for $A$ is exactly $(-A)^\beta,$ see \cite[p.187]{amansag2019maximal}. We have the following facts:
\begin{itemize}
  \item[{\rm (i)}]
 Assume that $H_\beta\subseteq \mathscr{Z}\subset H$ and $\mathscr{C}:\mathscr{Z}\to H$ is a linear operator and define $C:=\mathscr{C}|_{D(A)}$. Then $C$ satisfies the condition {\bf (A)}. In fact, let $\al>0$ and $x\in D(A)$. Then
\begin{align*}
\int^\al_0 \|CT(t)x\|^2dt&=\int^\al_0 \|\mathscr{C}T(t)x\|^2dt\cr & \le \|\mathscr{C}\|_{\calL(H_\beta,H)} \int^\al_0 \|(-A)^\beta T(t)x\|^2dt.
\end{align*}
Furthermore, we have $H_\beta\subset D(C_\Lambda)$ and  $C_\Lambda=\mathscr{C}$ on $H_\beta$.
  \item [{\rm (iii)}] Let $\theta\in (\frac{1}{2},1)$ another exponent. Take  $\mathscr{Z}:=H_\theta$, so that  $\mathscr{Z}\subset H_\beta\subset H$, due to $\theta>\beta$. Note that the operator $(-A)^\theta|_{D(A)}$ is not admissible for $A$, \cite{BDEH-21}. Thus a linear operator $\mathscr{C}:\mathscr{Z}\to H$ is not necessarily satisfies the condition {\bf(A)}.
\end{itemize}
 \end{remark}
In the rest of this section we are concerned with proving a regularity property of the following stochastic convolution for general semigroups
$$W_A^\Phi(t):=\int_0^t T(t-s)\Phi(s)dW(s).$$
In the case of an analytic semigroup  $\T$, we have the following classical result in stochastic analysis proved in Theorem 5.15 and Theorem 6.14 of  \cite{da2014stochastic}.
\begin{theorem}\label{classical-result}
	Let $\Phi(t)$, $t\geq0$ an $L_2^0(V,H)$-valued predictable process such that
	\begin{equation*}
		\mathbb{E}\int_0^t\|\Phi(s)\|^2_2ds<+\infty,\qquad t\geq 0.
	\end{equation*}
	\begin{itemize}
		\item[i)] If $A$ generates an analytic semigroup on $H$, then $W_A(t)\in D((-A)^\theta)$ for any $\theta\in(0,\frac{1}{2})$, $\mathbb{P}$-almost surely and
		$$\mathbb{E}\int_0^\alpha\left\|(-A)^\theta W_A^\Phi(t)\right\|^2dt\leq \gamma(\alpha)^2\mathbb{E}\int_0^\alpha\|\Phi(t)\|_2^2dt,$$
		for all $\alpha>0$ and some constant $\gamma(\alpha)>0$.
		\item[ii)] If $A$ generates a contractive and analytic semigroup on $H$, then $W_A^\Phi(t)\in D((-A)^\frac{1}{2})$, $\mathbb{P}$-almost surely and
		$$\mathbb{E}\int_0^\alpha\left\|(-A)^\frac{1}{2} W_A^\Phi(t)\right\|^2dt\leq \gamma(\alpha)^2\mathbb{E}\int_0^\alpha\|\Phi(t)\|_2^2dt,$$
		for all $\alpha>0$ and some constant $\gamma(\alpha)>0$.
	\end{itemize}
\end{theorem}

The following result shows a regularity property for the stochastic convolution without assuming any regularity on the semigroup $(T (t))_{t \geq0} $ and for  general unbounded linear operators $ C $ which are not necessarily closed or closable.
\begin{theorem}\label{6*}
Let $A$ and $C\in\calL(D(A),H)$ satisfy {\rm\bf(S)} and  {\rm\bf(A)}, and let  $(\Phi(t))_{t\ge 0}$ be an $L_2^0(V,H)$-valued predictable process such that
\begin{equation*}
\mathbb{E}\int_0^t\|\Phi(s)\|^2_2ds<+\infty,\qquad t\geq 0.
\end{equation*}
For all $\alpha>0,$
 \begin{align}\label{Reg-Max}
 \begin{split}
&  W_A^\Phi(t)\in D(C_{\Lambda}),\quad a.e.\, t\ge 0,\;\mathbb{P}-a.s,\quad and\cr & \mathbb{E}\int_0^\alpha\left\|C_{\Lambda} W_A^\Phi(t)\right\|^2dt\leq \gamma(\alpha)^2\mathbb{E}\int_0^\alpha\|\Phi(t)\|_2^2dt.
\end{split}
\end{align}
 \end{theorem}
\begin{proof}
Let $\mathcal{N}^2(0,\alpha;H)$ be the space of all predictable process $\Phi:[0,T]\times\Omega\to L_2^0(V,H)$ such that
\begin{align*}
\|\Phi\|_\alpha:=\mathbb{E}\int_0^\alpha\|\Phi(s)\|^2_2ds<+\infty.
\end{align*}
Let $\Phi\in \mathcal{N}^2(0,\alpha;H)$ and $\beta>\omega_0(A)$. By the same arguments as in \cite{lahbiri2020}, for any $\lambda\in (\beta,\infty)$,
\begin{align}
\mathbb{E}\int_0^\alpha \|C\lambda R(\lambda,A)W_A^\Phi(t)\|^2dt\leq \gamma(\alpha)^2 \||\lambda R(\lambda,A)\Phi(s)|\|_\alpha.\label{**}
\end{align}
This implies that  for $\lambda,\mu\in (\beta,\infty),$
\begin{align}
\begin{split}
\mathbb{E}\int_0^\alpha \|C\lambda &R(\lambda,A)W_A^\Phi(t)-C\mu R(\mu,A)W_A^\Phi(t)\|^2dt\cr & \leq \gamma(\alpha)^2 \|\lambda R(\lambda,A)\Phi(s)-\mu R(\mu,A)\Phi(s)\|_\alpha.\label{***}
\end{split}
\end{align}
Moreover,
\begin{align*}\|\lambda R(\lambda,A)&\Phi(s)-\mu R(\mu,A)\Phi(s)\|^2_2\cr & =\sum_{k=1}^{+\infty}\lambda_k\|\lambda R(\lambda,A)\Phi(s)e_k-\mu R(\mu,A)\Phi(s)e_k\|^2.\end{align*}
Since $\tr Q=\sum_{k=1}^{+\infty}\lambda_k<+\infty$, then $(\lambda R(\lambda,A)\Phi(s))_{\lambda>\beta}$ is a Cauchy sequence in $L_2^0(V,H)$. On the other hand, we have $\Phi\in \mathcal{N}^2(0,\alpha;H)$, and hence $(\lambda R(\lambda,A)\Phi(s))_{\lambda>\beta}$ is a Cauchy sequence in  $\mathcal{N}^2(0,\alpha;H)$ as well. Now, the inequality \eqref{***} implies that $(C\lambda R(\lambda,A)W_A^\Phi(\cdot))_{\lambda>\beta}$ is a Cauchy sequence in $L^2_\mathcal{F}([0,\alpha];H)$. Thus there exists $(\lambda_k)_{k\in \mathbb{N}}\subset\rho(A)\cap \mathbb{R}$ such that the  sequence $\left(C\lambda_kR(\lambda_k,A)W_A^\Phi(t)\right)_k$ converges for a.e $t\geq0$, $\mathbb{P}$-almost surely. Consequently, $W_A^\Phi(t)\in D(C_{\Lambda})$ for a.e $t\geq0$, $\mathbb{P}$-almost surely.
\end{proof}
\begin{remark}We can use Theorem \ref{6*}, to give an alternative proof to Theorem \ref{classical-result}. In fact, if we assume that $A$ generates an analytic semigroup $\T$ on $H$ and $\theta\in (0,\frac{1}{2}),$ then from Remark \eqref{zwart1}, we know that $C:=(-A)^\theta$ with domain $D(C)=D(A)$ satisfies the condition {\bf (A)} and $C_\Lambda=(-A)^\theta$. Thus  $W_A^\Phi(t)\in D\left((-A)^\theta\right) $. If in addition $\T$ is a contraction semigroup, then $C:=(-A)^\frac{1}{2}$ with domain $D(C)=D(A)$ satisfies the condition {\bf(A)}, due to  Le Merdy \cite{LeMery}. Thus,  Theorem \ref{6*} shows also that $W_A^\Phi(t)\in D((-A)^\frac{1}{2})$.
\end{remark}
In the following we will give an application to Theorem \ref{6*}. Let $\mathscr{B}:\mathscr{Z}\to H$ and $\mathscr{M}:H\to L^0_2(V,H)$ as in the introductory section.
\begin{definition}\label{mild-flandoli}
Let $A$ satisfies the condition {\bf(S)} and denote  $B:=\mathscr{B}$ with domain $D(B):=D(A)$. A process $X\in \calC_\calF([0,T],H)$ is called a  mild solution of the stochastic equation \eqref{Flandoli-exa} if the operator $B$ admits an extension $\tilde{B}:D(\tilde{B})\subset H\to H$ such that $X(t)\in D(\tilde{B})$ for a.e. $t>0$, $\P$-a.s, $X(t)$ is $\calF_t$-adapted for any $t\in [0,T],$ the maps $t\mapsto X(t)$ and $t\mapsto \tilde{B}X(t)$ are measurable, $2$-integrable on $\Om\times [0,T]$, and satisfies the integral equation
\begin{align*}
X(t)=T(t)\xi+\int^t_0 T(t-s)\mathscr{M}(\tilde{B} X(s))dW(s)
\end{align*}
for any $\xi\in L^2(\mathcal{F}_0,H),$ and  $t\ge 0$, $\P$-a.s.
\end{definition}

The following result introduce conditions for the existence and uniqueness of the mild solution of the semilinear stochastic equation \eqref{Flandoli-exa}.

\begin{theorem}\label{B-theorem}Assume that $A$ satisfies {\bf(S)}, $\mathscr{M}$ satisfies the condition {\bf(L)}, and the unbounded linear operator $B$ satisfies the admissibility condition {\bf(A')}. Then there exists a unique $X\in \calC_\mathcal{F}([0,T],H)$ of the stochastic equation \eqref{Flandoli-exa} such that $X(t)\in D(B_\Lambda)$ for a.e. $t>0$ and $\P$-a.s., and
\begin{align*}
X(t)=T(t)\xi+\int^t_0 T(t-s)\mathscr{M}(B_\Lambda X(s))dW(s)
\end{align*}
for any $t\ge 0$ and $\xi\in L^2(\mathcal{F}_0,H),$ where $B_\Lambda$ is the Yosida extension of $B$ for $A$.
\end{theorem}
\begin{proof}
Let $u\in L^2_\calF([0,\al],H)$ for $\al>0$ and consider the process
\begin{align} \label{var}
X(t;u)=T(t)\xi+\int^t_0 T(t-s)\mathscr{M}(u(s))dW(s),\quad t\in [0,\al],\quad \xi\in L^2(\mathcal{F}_0,H).
\end{align}
According to \eqref{C-Weiss} and Theorem \ref{6*}, $X(t)\in D(B_\Lambda)$ for $t>0$, $\P$-a.s., and there exist constants $c>0$ and $\tilde{\ga}>0$ such that
\begin{align*}
\E\int^\al_0 \|B_\Lambda X(t)\|^2 dt\le c \tilde{\ga}(\al)\left(\E\|\xi\|^2+\E \|u\|_{L^2([0,\al],H)}\right)
\end{align*}
for any $\xi\in L^2(\mathcal{F}_0,H)$ and $u\in L^2_\calF([0,\al],H)$, where $\tilde{\ga}(\al)\to 0$ as $\al\to 0$. We now select the following map
\begin{align*}
\Gamma: L^2_\calF([0,\al],H)\to L^2_\calF([0,\al],H),\quad (\Gamma u)(t)=B_\Lambda X(t;u).
\end{align*}
For any $u_1$ and $u_2$ in $L^2_\calF([0,\al],H),$ we have
\begin{align*}
\|\Gamma u_1 -\Gamma u_2\|_{L^2_\calF([0,\al],H)}\le \tilde{\ga}(\al)\|u_1-u_2\|_{L^2_\calF([0,\al],H)}.
\end{align*}
We can choose $\al_0>0$ such that $0<\tilde{\ga}(\al_0)<1$, so that $\Gamma$ is a contraction on $L^2_\calF([0,\al],H)$. By the Banach's fixed pint theorem, there is a unique $u\in L^2_\calF([0,\al],H)$ such $u=\Gamma u=B_\Lambda X(\cdot,u)$. Now replacing $u$ in \eqref{var}, we have
\begin{align*}
X(t)=T(t)\xi+\int^t_0 T(t-s)\mathscr{M}(B_\Lambda X(s))dW(s),\quad t\in [0,\al_0],\; \xi\in L^2(\mathcal{F}_0,H).
\end{align*}
By standard arguments the
restriction on $\al_0$ can be removed.
\end{proof}
\begin{remark}
\begin{itemize}
  \item[{\rm(i)}] The linear case of the equation \eqref{Flandoli-exa} (i.e. $\mathscr{M}=I$) is considered in \cite[Section 6.5, page 176]{da2014stochastic}, where additional conditions on the semigroup $\T$ and the perturbation $B$ are assumed. In fact, in this reference the authors are mainly based on the concept of maximal regularity and interpolation spaces to introduce an appropriate  condition on $B$. These facts facilities the use of the classical Banach's fixed theorem to prove the well-posedness of the equation as well as the representation of the solution in terms of a variation of constants formula. Compared with our result, we have less conditions on $B$ and on the generator $A$, due to the regularity of the stochastic convolution obtained in Theorem \ref{6*}.
\item[{\rm(ii)}]In the case of $\mathscr{M}\equiv I$ , the equation \eqref{Flandoli-exa} was also treated by Bonaccorsi \cite{bouna} Flandoli \cite[Theorem 1.2]{Flandoli1} where the semigroup $(T(t))_{t\geq0}$ satisfies a regularity condition and the operator $B$ is (somehow) zero-class admissible for $A$. Moreover, in \cite{bouna}, the author used Malliavin calculus to represent the solution via an appropriate variation of constants formula. In Theorem \ref{B-theorem} we do not assume any regularity on the semigroup $(T(t))_{t\geq0}$.
  \item[{\rm(iii)}] In \cite{Hos-16}, the author considered the KPZ equation driven by space-time white
noise replaced with its fractional derivatives of order $\ga>0$ in spatial variable. This equation has the same regularity as the solution of the equation \eqref{Flandoli-exa} with $A=\partial_x^2$, $\mathscr{M}=I$ and $B=\partial_x^\sigma$. According to remark \ref{zwart1}, $\partial_x^\sigma$ is a zero-class admissible operator for $\partial_x^2$ for any $0<\sigma<1$.
\end{itemize}

\end{remark}

\begin{example}\label{sec2:ex1}
Let $\mathscr{O}\subset \mathbb{R}^n$ be an open bounded set with a $C^2$ boundary $\partial\mathscr{O}$ and outer unit normal $\nu$ and put $H=L^2(\mathscr{O})$ and $\mathscr{Z}=H^2(\mathscr{O})$. We consider the following nonlinear initial value problem
\begin{align}\label{exam2}
\begin{split}
		& \frac{\partial}{\partial t} X(t,x)=\Delta X(t,x)+f\left(\displaystyle\int_{\partial\mathscr{O}} \Upsilon(x,y)c_1(y)X(t,y)dy\right) \frac{\partial}{\partial t} W(t,x)\cr & \hspace{8cm} x\in \mathscr{O}, t\in[0,T],\\
		& X(0,x)=g(x), \hspace{5cm} x\in \mathscr{O}, \xi\in L^2(\mathscr{O}),\\  &  \nabla X(t,x)|\nu(x)=c_2(x)X(t,x), \hspace{3cm} x\in \partial \mathscr{O}, t\in[0,T],
\end{split}
\end{align}
where $\Upsilon\in L^\infty(\mathscr{O}\times \partial \mathscr{O})$, $c_1,c_2\in C_b(\partial\mathscr{O})$, $f$ is a real valued function, globally Lipschitz and $W(t)$ is $Q$-Wiener process on $L^2(\mathscr{O})$.
Define the following operators
\begin{align*}
& A=\Delta,\quad D(A)=\left\{\varphi\in H^2(\mathscr{O}): \nabla \varphi(x)|\nu(x)=c_2(x)\varphi(x),\quad x\in\partial\mathscr{O}\right\},\cr & (\mathscr{M}(\phi))(x)=f(\phi(x)),\qquad x\in\mathscr{O},\cr & (\mathscr{R}\varphi)(x)= \int_{\partial\mathscr{O}} \Upsilon(x,y)\varphi(y)dy,\qquad x\in\mathscr{O},\cr
	& (\Theta\varphi)(y)=c_1(y)\varphi(y),\quad y\in\partial\mathscr{O},\quad \varphi\in L^2(\partial\mathscr{O})\cr & \mathscr{B}= \mathscr{R} \Theta: H^2(\mathscr{O})\to L^2(\mathscr{O}), \cr & \mathscr{K}:H^2(\mathscr{O})\to L^2(\partial\mathscr{O}),\quad (\mathscr{K}\varphi)(x)=c_2(x)\varphi(x),\quad x\in\partial\mathscr{O}.
\end{align*}
It suffices to prove that $A$ is a generator of a strongly continuous semigroup $\T$ on $H,$ and $B:=\mathscr{B}$ with domain $D(B)=D(A)$ is a zero-class admissible operator for $A$.  To this end, we will use a perturbation result in \cite{HMR-2015}. In fact, the following operator
$$ A_0:=\Delta,\quad D(A_0)=\left\{\varphi\in H^2(\mathscr{O}): \nabla \varphi(x)|\nu(x)=0,\quad x\in\partial\mathscr{O}\right\}$$
generates an analytic semigroup $\T_0:=(\T_0(t))_{t\ge 0}$ on $H$. Now denote by $\varphi=\mathscr{N}\psi\in H^{\frac{3}{2}}(\mathscr{O})$ the solution of the elliptic boundary value problem $ \varphi=0$ on $\mathscr{O}$ and  $\nabla \varphi|\nu=\psi$ on $\partial\mathscr{O}$ for $\psi\in L^2(\partial\mathscr{O})$. Then $\mathscr{N}$ is continuous from $L^2(\partial\mathscr{O})$ to $D((-A_0)^\beta)$ for any $\beta\in (0,\frac{3}{4})$. We select the operators
\begin{align*}
\Xi:=-A_0 \mathscr{N}:L^2(\partial\mathscr{O})\to D(A_0^\ast)',\cr K:=\mathscr{K},\qquad D(K):=D(A_0).
\end{align*}
We will verify that the triple of operators $(A_0,\Xi,K)$ satisfies the condition of \cite[Theorem 4.1]{HMR-2015}. In fact, for
for any $t>0,$ $\theta\ge 0$ and $0<\beta<3/4$
	\begin{equation}\label{hadd}
\begin{split}
		 (-A_0)^\theta \T_0(t)\Xi\in \calL(L^2(\partial\mathscr{O}),H), \quad \|(-A_0)^\theta \T_0(t)\Xi\| \le \kappa_\beta t^{\beta-\theta-1},
\end{split}
	\end{equation}
	where $\kappa_\beta>0,$ is a constant, due to \cite{LT}. By choosing $\beta\in (\frac{1}{2},\frac{3}{4})$ and $\theta=0$ (hence $2(1-\beta)<1$), we obtain
	\begin{align*}
		\left\|\int^t_0 \T_0(t-s)\Xi v(s)ds\right\|_{L^2(\Om)}\le \delta_\beta \|\|v\|_{L^2([0,t],L^2(\mathscr{O}))}
	\end{align*}
	for all $t>0$ and $v\in L^2([0,t],L^2(\mathscr{O})),$ where $\delta_\beta>0$ is a constant, due to \eqref{hadd}. In the terminology of control theory, this means that $\Xi$ is an admissible control operator for $A_0$. On the other hand,  the operator $\mathscr{K}:D((-A_0)^\theta)\to L^2(\mathscr{O})$ is uniformly bounded for any $\theta>\frac{1}{4}$, so that $\mathscr{K}(-A_0)^{-\theta}\in \calL(L^2(\mathscr{O}))$ and $\|\mathscr{K}(-A_0)^{-\theta}\|\le \eta$ for a constant $\eta>0$.  Let us now choose $\theta\in (\frac{1}{4},\frac{1}{2})$. We have
	\begin{align}\label{E2}
		\|Ke^{tA_0}\|\le \eta \frac{M_0}{t^{\theta}},\quad (t>0).
	\end{align}
This shows that $K$ satisfies \eqref{adm} with respect to $\T_0$ and a constant $\ga_0(\alpha)\to 0$ as $\alpha\to 0$. On the other hand, by using Remark \ref{zwart1}, the Yosida extension of $K$ for $A_0$ is exactly $K_\Lambda =\mathscr{K}$. On the other hand, by using \eqref{hadd} and the fact that then $\frac{1}{2}<1-(\beta-\theta)<1$, we obtain
	\begin{align}\label{E3}
		\left(\int^\al_0\left\|\mathscr{K}\int^t_0 \T_0(t-s)\Xi v(s) ds\right\|^2 dt\right)^{\frac{1}{2}} \le \eta \kappa_\beta  \al^{\beta-\theta} \|v\|_{L^2([0,\al],L^2(\partial\mathscr{O}))}
	\end{align}
	for any $v\in L^2(\partial\mathscr{O})$. This implies  that the operator $A$ coincides with the following one
\begin{align*}
A_0^{cl}:=A_0+\B \mathscr{K},\quad D(A_0^{cl})=\{\varphi\in H^2(\mathscr{O}):(A_0+\Xi \mathscr{K})\varphi\in H\},
\end{align*}
which generates a strongly continuous semigroup $\T=(\T(t))_{t\ge 0}$ on $H$ such that
\begin{align}\label{SM}
\T(t)g=\T_0(t)g+\int^t_0 \T_0(t-s)\Xi \mathscr{K} \T(s)gds,\qquad (t\ge 0,\;g\in H),
\end{align}
see the proof of Theorem 4.1 of \cite{HMR-2015}. According to \cite{Weiss-regular}, if we define $K^A:=\mathscr{K}$ with domain $D(K^A)=D(A)$ then its Yosida extension for $A$ is $K^A_\Lambda=K_\Lambda=\mathscr{K}$, and \begin{align}\label{K-A-admissible} \int^\al_0 \|\mathscr{K}\T(t)g\|^2dt\le \tilde{\gamma}^2 \|g\|^2_H\end{align} for any $g\in D(A)$ and $\al>0$, and a constant $\tilde{\ga}:=\tilde{\ga}(\al)>0$. Let us now prove that $B:=\mathscr{B}$ with domain $D(B)=D(A)$ is a zero-class admissible operator for $A$. To this end, we select $B_0:=\mathscr{B}$ with domain $D(B_0)=D(A_0)$. As the operator $\mathscr{B}$ and $\mathscr{K}$ have the same form, then they have the same properties. So that $B_0$ satisfies the condition \eqref{E2} if we replace $K$ by $B_0$. Also $\mathscr{B}$ satisfies the inequality \eqref{E3} if we replace $\mathscr{K}$ by $\mathscr{B}$. Then for any $g\in D(A)$ and $t\ge 0,$
\begin{align*}
B\T(t)g= \mathscr{B}\T(t)g&=\mathscr{B}\left(\T_0(t)g+\int^t_0 \T_0(t-s)\Xi \mathscr{K} \T(s)gds\right)\cr &= B_0\T_0(t)g+\mathscr{B}\int^t_0 \T_0(t-s)\Xi \mathscr{K} \T(s)gds
\end{align*}
Remak that \eqref{E3} holds also if we replace $\mathscr{K}$ by $\mathscr{B}$. Now by using this fact and the estimate \eqref{K-A-admissible}, for any $\al>0$, there exists a constant $\kappa>0$ independent of $\al>0$ such that
 \begin{align*}
 \int^\al_0 \|B\mathbb{T}(t)g\|^2dt\le \kappa \left( \al^{1-2\theta}+ \al^{2(\beta-\theta)} \right) \|g\|^2_H \qquad (g\in D(A)).
 \end{align*}
 This $B$ is a zero-class admissible operator for $A$, so it satisfies the condition {\bf(A')}. Let $B_\Lambda$ be the Yosida extension of $B$ for $A$, and let us prove that $B_\Lambda=\mathscr{B}$. In fact, for a large $\la>0$ and by taking Laplace transform on the both sides of \eqref{SM}, we obtain
 \begin{align*}
 R(\la,A)=R(\la,A_0)+R(\la,A_0)\Xi \mathscr{K} R(\la,A)
 \end{align*}
 As for $K,$ the Yosida extension of $B_0$ for $A_0$ is $(B_0)_\Lambda=\mathscr{B}$. By using the fact $\calB R(\la,A_0)\Xi\to 0$ as $\la\to 0$ (see \cite{Weiss-regular}), for any $g\in H^2(\mathscr{O}),$
 \begin{align*}
 \|\la \mathscr{B}R(\la,A_0)\Xi \mathscr{K} R(\la,A)g\| \le \|\calB R(\la,A_0)\Xi\| \left(\| \mathscr{K} \la R(\la,A)g-\mathscr{K}g\|+\|\mathscr{K}g\|\right).
 \end{align*}
 This implies that for any $g\in H^2(\mathscr{O}),$
 \begin{align*}
 \lim_{\la\to+\infty}B \la R(\la,A)g=\lim_{\la\to+\infty}\mathscr{B}\la R(\la,A_0)g=\mathscr{B}g.
 \end{align*}
 Thus by Theorem \ref{B-theorem}, the solution of the equation \eqref{exam2}  satisfies $X(t)\in H^2(\mathscr{O})$ for a.e. $t>0$, $\mathbb{P}$-a.s. and
 \begin{align*}
 X(t)=\T(t)g+\int^t_0 \T(t-s)\mathscr{M}(\mathscr{B}X(s))dW(s)
 \end{align*}
 for any $t\ge 0,$ and $g\in H$.
\end{example}

\section{The well-posedness}\label{sec:well-posed}

In this section we study the well-posedness of the semilinear stochastic equation \eqref{0} (hence of \eqref{****}). Due to the form of the nonlinear term $G$ given in \eqref{G}, it is not clear how to define the mild solution $X(\cdot)$ of the stochastic equation \eqref{0}. This is because the operator $\mathscr{C}$ is only defined on a small domain of $H$. However, we will adopt a kind of solution to the equation \eqref{0} that will be justified in the main result of this section.
\begin{definition}\label{def3.1}
Let the condition {\bf(S)} be satisfied. A continuous process $X(\cdot):[0,T]\mapsto H$ is a mild solution of \eqref{0} if there exists an operator $\tilde{\mathscr{C}}:D(\tilde{\mathscr{C}})\subset H\to H$, extension of $\mathscr{C}$  such that
\begin{enumerate}
\item $X(t)$ is $\mathcal{F}_t$-adapted, for each $0\leq t\leq T$, $X(t)\in D(\tilde{\mathscr{C}})$ a.e $t\geq0,$ $\mathbb{P}$-almost surely,
\item $X(t)$ and $\tilde{\mathscr{C}} X(t)$ are measurable,  and
$$\int_0^T\|X(t)\|^2dt<\infty, \qquad \int_0^T\|\tilde{\mathscr{C}} X(t)\|^2dt<\infty,\quad \mathbb{P}\text{-a.s},$$
\item $X(\cdot)$ satisfies the stochastic integral equation
\begin{align*}
X(t)=T(t)\xi+\int_0^tT(t-s)F(\tilde{\mathscr{C}} X(s))ds
 +\int_0^tT(t-s)\mathscr{M}(X(s))dW(s),
\end{align*}
for all $t\in[0,T]$, $\mathbb{P}$-almost surely.
\end{enumerate}
\end{definition}
The following technical results are needed to prove the main results of this sections.
\begin{lemma}
Let assumptions {\bf (S)}, {\bf (A)} and {\bf (L)} be satisfied and let $\xi\in L^2(\mathcal{F}_0,H)$. For all $u\in \mathcal{C}_{\mathcal{F}}([0,T];H)$ there exists a sequence $(X_n)_{n\in\mathbb{N}}\subset \mathcal{C}_{\mathcal{F}}([0,T];H)$ such that for all $n\in\mathbb{N}$
$$X_n(t)\in D(C_\Lambda)\quad \text{a.e.} ~t\geq0, \mathbb{P}-\text{almost surely},~\text{and}$$
	\begin{align}\label{lemme3}
		\begin{split}
& X_0(t)=T(t)\xi+\int_0^t T(t-s)\mathscr{M}(u(s))dW(s), \cr
& X_{n+1}(t)=T(t)\xi+\int_0^t T(t-s)F(C_{\Lambda}X_n(s))ds\cr & \hspace{5cm}+\int_0^t T(t-s)\mathscr{M}(u(s))dW(s).
		\end{split}
	\end{align}
\end{lemma}
\begin{proof}
	Let $\xi\in L^2(\mathcal{F}_0,H)$ and $u\in \mathcal{C}_{\mathcal{F}}([0,T];H)$. Using the condition ${\bf (A)}$ and \eqref{C-Weiss}, we have $X_0(t)\in D(C_\Lambda)$ for a.e $t\geq0$ and $\mathbb{P}$-a.s, and
	\begin{align}\label{C-estimate}
		\E\int^\al_0 \|C_\Lambda X_0(s)\|^2ds<\infty,
	\end{align}
	for any $\alpha>0$. This allows us to define the following process
	\begin{align*}X_1(t)=T(t)\xi+\int_0^t T(t-s)F(C_{\Lambda}X_0(s))ds+\int_0^t T(t-s)\mathscr{M}(u(s))dW(s),\end{align*}
	for all $u\in \calC_\mathcal{F}([0,T];H)$.  It follows from \eqref{C-estimate}, Proposition \ref{5*} and Proposition \ref{6*} that $X_1(t)\in D(C_{\Lambda})$ a.e $t\geq0$, $\mathbb{P}$-a.s., and
	\begin{align*}
		\E\int^\al_0 \|C_\Lambda X_1(s)\|^2ds<\infty,
	\end{align*}
	for any $\alpha>0$. Therefore, by induction, $X_n(t)\in D(C_{\Lambda})$ a.e $t\geq0$, $\mathbb{P}$-a.s., and
	\begin{align*}
		\E\int^\al_0 \|C_\Lambda X_n(s)\|^2ds<\infty,
	\end{align*}
	and the following sequence
	$$X_{n+1}(t)=T(t)\xi+\int_0^t T(t-s)F(C_{\Lambda}X_n(s))ds+\int_0^t T(t-s)\mathscr{M}(u(s))dW(s),$$
	is well defined for any $n\in \mathbb{N}$.
\end{proof}

\begin{lemma}\label{lemme3.3}
Let assumptions {\bf (S)}, {\bf (A)} and {\bf (L)} be satisfied. For any process  $u\in \mathcal{C}_{\mathcal{F}}([0,T];H)$ there exists a unique $X(\cdot)\in \mathcal{C}_{\mathcal{F}}([0,T];H)$ such that
\begin{gather}\label{3.4}
	X(t)\in D(C_{\Lambda}), \quad \text{a.e}~t\geq 0, ~\mathbb{P}\text{-almost surely},\notag\\
	X(t)=T(t)\xi+\int_0^tT(t-s)F(C_{\Lambda}X(s))ds+\int_0^tT(t-s)\mathscr{M}(u(s))dW(s).
\end{gather}

\end{lemma}
\begin{proof}
	Let $X_n$ be as in \eqref{lemme3}. Using  \eqref{C-Weiss} and Proposition \ref{5*}, we have
	\begin{align*}
		&\mathbb{E}\int_0^t\|C_{\Lambda}X_n(s)-C_{\Lambda}X_{n-1}(s)\|^2ds\cr  & \qquad\qquad\leq c(t)^2k^2\mathbb{E}\int_0^t\left\|C_{\Lambda}X_{n-1}(\sigma)-C_{\Lambda}X_{n-2}(\sigma)\right\|^2d\sigma.
	\end{align*}
	for all $t\in[0,T]$. By induction, we get
	\begin{align*}
		\mathbb{E}\int_0^t\|C_{\Lambda}X_n(s)-C_{\Lambda}X_{n-1}(s)\|^2ds\leq \left(c(t)k\right)^{2n}\gamma(t)^2\mathbb{E}\|\xi\|^2.
	\end{align*}
Since $\gamma(t)$ is nondecreasing in $t$, it follows that $\gamma(t)\leq \gamma(T)$ and $c(t)\leq c(T) $ for any $t\in [0,T]$. Thus, for any $t\in [0,T]$,
\begin{equation}\label{3.3}
	\mathbb{E}\int_0^T\|C_{\Lambda}X_n(s)-C_{\Lambda}X_{n-1}(s)\|^2ds\leq \left(c(T)k\right)^{2n}\gamma(T)^2\mathbb{E}\|\xi\|^2.
\end{equation}
Moreover, $\lim_{t\to 0}c(t)=0$, then we can choose $T$ small enough such that $c(T)k<1$ and consequently, $(C_{\Lambda}X)_n$ is a Cauchy sequence in $L^2_{\mathcal{F}}([0,T];H)$. The case of general $T>0$ can be treated by considering the equation in intervals $[0,T_0]$, $[T_0,2T_0]$,$\dots$ with $T_0$ such that $c(T_0)k<1$. On the other hands, simple calculations lead to
$$\sup_{t\in[0,T]}\mathbb{E}\|X_{n+1}-X_n(t)\|^2\leq TM^2e^{2|\beta|T}k^2\int_0^T\|C_{\Lambda}X_n(s)-C_{\Lambda}X_{n-1}(s)\|^2ds.$$
From \eqref{3.3}, it follows that $(X_n)_n$ is a Cauchy sequence in $\calC_\mathcal{F}([0,T];H)$.\\
Now, let
\begin{align*}
	X:=\lim_{n\to\infty}X_n~~ \text{in}~~ \calC_\mathcal{F}([0,T];H),\\
	Y:=\lim_{n\to\infty}C_{\Lambda}X_n~~ \text{in}~~ L^2_{\mathcal{F}}([0,T];H).
\end{align*}
It is clear that $\int_0^t T(t-s)F(C_{\Lambda}X_n(s))ds$ converges to $\int_0^t T(t-s)F(Y(s))ds$ as $n\to \infty$ for all $t\in [0,T]$. Thus,
$$X(t)=T(t)\xi+\int_0^t T(t-s)F(Y(s))ds+\int_0^t T(t-s)\mathscr{M}(u(s))ds,$$
and $X(t)\in D(C_{\Lambda})$ for $t\geq0$ a.e, $\mathbb{P}$ almost surely. Moreover, we have
$$\mathbb{E}\int_0^T\|C_{\Lambda}X_n(s)-C_{\Lambda}X(s)\|^2\leq c(T)^2k^2\mathbb{E}\int_0^T\|C_{\Lambda}X_n(s)-Y(s)\|^2ds,$$
for all $T\geq0$. It follows that $Y(t)=C_{\Lambda}X(t)$ and
\begin{equation*}
	X(t)=T(t)\xi+\int_0^t T(t-s)F(C_{\Lambda}X(s))ds+\int_0^t T(t-s)\mathscr{M}(u(s))dW(s)
\end{equation*}
for all $t\in [0,T]$.
\end{proof}
\begin{theorem}\label{thm1}
Let assumptions {\bf (S)}, {\bf (A)} and {\bf (L)} be satisfied. For any initial process $\xi\in L^2(\mathcal{F}_0,H),$ the semilinear stochastic equation \eqref{0} has a unique mild solution $X(\cdot)\in\calC_\mathcal{F}([0,T];H)$ such that
\begin{gather*}
	X(t)\in D(C_{\Lambda}), \quad \text{a.e}~t\geq 0, ~\mathbb{P}\text{-almost surely},\\
	X(t)=T(t)\xi+\int_0^tT(t-s)F(C_{\Lambda}X(s))ds+\int_0^tT(t-s)\mathscr{M}(X(s))dW(s).
\end{gather*}
\end{theorem}

\begin{proof}
Define the map $\mathscr{R}:\mathcal{C}_{\mathcal{F}}([0,T];H)\to \mathcal{C}_{\mathcal{F}}([0,T];H)$ by
\begin{align*}
	(\mathscr{R} u)(t):=X(t,u),
\end{align*}
where $X(t,u)$ is given by \eqref{3.4}. It suffices to prove that the map $\mathscr{R}$ has a unique fixed point in  $\mathcal{C}_{\mathcal{F}}([0,T];H)$. To this end, consider $u,v\in \mathcal{C}_{\mathcal{F}}([0,T];H)$. A similar calculations as in the proof of  Lemma \ref{lemme3.3} lead to
\begin{align*}
	&\mathbb{E}\int_0^t\|C_{\Lambda}X(s,u)-C_{\Lambda}X(s,v)\|^2ds \cr
	&\qquad\leq 2c(T)^2k^2\mathbb{E}\int_0^t\|C_{\Lambda}X(s,u)-C_{\Lambda}X(s,v)\|^2ds\\
	&\qquad\qquad +2\gamma(T)^2k^2\mathbb{E}\int_0^t\|u(s)-v(s)\|^2ds,
\end{align*}
for any $t\in[0,T]$. We have $\lim_{t\to0}c(t)=0$, then we can choose $T>0$ such that $2c(T)^2k^2<\frac{1}{2}$, this implies that
\begin{align*}
	&\mathbb{E}\int_0^t\| C_{\Lambda}X(s,u)-C_{\Lambda}X(s,v)\|^2ds \leq 4k^2T\gamma(T)^2\sup_{t\in[0,T]}\mathbb{E}\|u(t)-v(t)\|^2.
\end{align*}
Consequently,
\begin{align*}
	&\sup_{t\in[0,T]}\mathbb{E}\|\left(\mathscr{R}u\right)(t)-\left(\mathscr{R}v\right)(t)\|^2\cr
	&\qquad \leq 2M^2k^2e^{2|\beta|T}T\mathbb{E}\int_0^t\|C_{\Lambda}X(s,u)-C_{\Lambda}X(s,v)\|^2ds\\
	&\qquad\qquad +2M^2k^2e^{2|\beta|T}\mathbb{E}\int_0^t\|u(s)-v(s)\|^2ds\\
	&\qquad \leq \zeta_T\sup_{t\in[0,T]}\mathbb{E}\|u(t)-v(t)\|^2,
\end{align*}
where $\zeta_T:=8M^2k^4e^{2|\beta|T}\gamma(T)^2T^2+2M^2k^2e^{2|\beta|T}T$. We choose $T$ small enough such that $T=max\{T>0, \quad 2c(T)^2k^2<\frac{1}{2}~~\text{and}~~\zeta_T<1 \}$ and consequently, by the contraction principle,  there exists a unique $v\in \mathcal{C}_{\mathcal{F}}([0,T];H)$ such that
$$v(t)=T(t)\xi+\int_0^tT(t-s)F(C_{\Lambda}X(s,v))ds+\int_0^tT(t-s)\mathscr{M}(v(s))dW(s).$$
The case of general $T>0$ can be treated by considering the equation in intervals $[0,T_0]$, $[T_0,2T_0]$,$\dots$ with $T_0$ such that $T_0=max\{T_0>0, \quad 2c(T_0)^2k^2<\frac{1}{2}~~\text{and}~~\zeta_{T_0}<1 \}$. On the other hands, from Lemma \ref{lemme3}, for all $u\in\mathcal{C}_{\mathcal{F}}([0,T];H)$ there exists a unique $X(\cdot)\in \mathcal{C}_{\mathcal{F}}([0,T];H)$ such that \eqref{3.4} is satisfied. Thus, by the uniqueness, there exists a unique $X(\cdot)\in \mathcal{C}_{\mathcal{F}}([0,T];H)$ such that
\begin{align*}
X(t)=T(t)\xi+\int_0^tT(t-s)F(C_{\Lambda}X(s))ds+\int_0^tT(t-s)\mathscr{M}(X(s))dW(s).
\end{align*}
\end{proof}

We end this section by giving some examples that satisfy our abstract results. In the following we give two examples, the first example is devoted to kind of Kardar-Parizi-Zhang equation. The second is devoted to the Schr\"odinger equation with non-local integral term.
\begin{example}
Let $\mathscr{O}\subset \mathbb{R}^n$ be an open bounded set and put $H=L^2(\mathscr{O})$. We consider the following nonlinear initial value problem
\begin{eqnarray}\label{exx}
	\begin{cases}
		dX(t,x)=\left[\Delta X(t,x)+f\left((a.\nabla+c)X(t,x)\right)\right]dt\cr \hspace{3cm}+b(X(t,x))dW(t,x), & x\in \mathscr{O}, t\in[0,T],\\
		X(t,x)=0, & x\in \partial \mathscr{O}, t\in[0,T],\\
		X(0,x)=\xi(x), & x\in \mathscr{O}, \xi\in L^2(\mathscr{O}),
	\end{cases}
\end{eqnarray}
where $a\in L^\infty(\mathscr{O};\mathbb{C}^n)$, $c\in L^\infty(\mathscr{O})$ and $W(t)$ is $Q$-Wiener process on $V=H=L^2(\mathscr{O})$. We assume that $f$, $b$ are real valued functions and globally Lipschitz.  Define $A:D(A)\to H$ to be
\begin{align*}A\phi=\Delta\phi,\qquad D(A)=H^2(\mathscr{O})\cap H^1_0(\mathscr{O}).\end{align*} So  that $-A$ is a strictly
positive densely defined operator on $H$. It is known that $A$ generates a strongly continuous and diagonalizable  semigroup $(T(t))_{t\geq 0}$ on $H$. We select $H_{\frac{1}{2}}=D((-A)^{\frac{1}{2}})=H^1_0(\mathscr{O})$. We now define the following operators
\begin{align*}
	& \mathscr{C}z=a\cdot\nabla z+cz,\quad z\in H_{\frac{1}{2}},\cr
	& (F\phi)(x)=f(\phi(x)), \quad x\in \mathscr{O},\cr
	& (\mathscr{M}(\phi)u)(x)=b(\phi(x))u(x), \quad x\in \mathscr{O},\quad u\in H.
\end{align*}
Thus, the system \eqref{exx} becomes
\begin{eqnarray*}
	\begin{cases}
		dX(t)=\left[AX(t)+F\left(\mathscr{C}X(t)\right)\right]dt+\mathscr{M}(X(t))dW(t), &t\in[0,T],\\
		X(0)=\xi.
	\end{cases}
\end{eqnarray*}
Moreover, it is shown in \cite[Example 5.1.4]{tucsnak2009observation} that $C:=\mathscr{C}$  with domain $D(C)=D(A)$ satisfies the condition \textbf{(A)}. Also as in Remark \ref{zwart1} we have $C_\Lambda=\mathscr{C}$.
So, for each $\xi\in L^2(\mathcal{F}_0,H)$, \eqref{exx} has a unique mild solution  $X(\cdot)\in \calC_\mathcal{F}([0,T];H)$ such that
\begin{gather*}
	X(t)\in H_{\frac{1}{2}}, \quad\text{a.e}~ t\geq0, \mathbb{P}\text{-almost surely},\\
	X(t)=T(t)\xi+\int_0^t T(t-s)F(\mathscr{C}X(s))ds+\int_0^tT(t-s)\mathscr{M}(X(s))dW(s).
\end{gather*}
\end{example}

\begin{example}
Let $\Omega$ be an open bounded domain in $\mathbb{R}^n$, $n\geq2$ with sufficiently smooth boundary $\partial \Omega=\overline{\Gamma_0\cup\Gamma_1}$, $\Gamma_0\cap\Gamma_1=\emptyset$. We consider the following semilinear stochastic Schrodinger equation with non-local integral term
\begin{align}\label{exxx}
	\begin{cases}
		dX(t,x)=\left[i\Delta X(t,x)+f\left(i\displaystyle\int_{\Gamma_1}\Upsilon(x,y)X(t,y)dy\right)\right]dt\cr \hspace{3cm}+b(X(t,x))dW(t,x), & x\in \Omega, t\in[0,T],\\
		X(t,x)=0, & x\in \partial \Gamma_0, t\in[0,T],\\
		\frac{\partial X}{\partial \nu}(t,x)=0, & x\in \partial \Gamma_1, t\in[0,T],\\
		X(0,x)=\xi(x), & x\in \Omega, \xi\in L^2(\Omega),
	\end{cases}
\end{align}
where $\nu$ is the unit normal vector,   $\Upsilon: \overline\Omega\times\partial\Omega \to \mathbb{R}$ is a continuous function and $W(t)$ is $Q$-Wiener process on $L^2(\Omega)$. As usual, we write the equation  in an abstract form. We introduce the space $H=L^2(\Omega)$. We define the operators
\begin{align*}
	&A=i\Delta, \quad D(A)=\left\{g\in H^2(\Omega), \qquad g_{|\Gamma_0}=\left.\frac{\partial g}{\partial\nu}\right|_{\Gamma_1}=0 \right\},\cr
	&\mathscr{C}g=i\int_{\Gamma_1}\Upsilon(\cdot,y)g(y)dy, \quad g\in H^2(\Omega),\cr
	& (F\phi)(x)=f(\phi(x)), \quad x\in \Omega,\cr
	& (\mathscr{M}(\phi)u)(x)=b(\phi(x))u(x), \quad x\in \Omega,\quad u\in H,
\end{align*}
and rewrite \eqref{exxx} as
\begin{eqnarray*}
	\begin{cases}
		dX(t)=\left[AX(t)+F\left(\mathscr{C}X(t)\right)\right]dt+\mathscr{M}(X(t))dW(t), &t\in[0,T],\\
		X(0)=\xi.
	\end{cases}
\end{eqnarray*}
It is well known that the operator $A$ generates an unitary group $(T(t))_{t\geq0}$ on $H$. Moreover, it is shown in \cite[Example 4.2]{Lasri2021} that the operator $C:=\mathscr{C}$  with domain $D(C)=D(A)$ satisfies the condition \textbf{(A)}. Therefore, by Theorem \ref{thm1}, the equation \eqref{exxx} has unique mild solution $X(\cdot)\in \mathcal{C}_\mathcal{F}([0,T];H)$ such that
\begin{gather*}
	X(t)\in D(C_\Lambda), \quad\text{a.e}~ t\geq0, \mathbb{P}\text{-almost surely},\\
	X(t)=T(t)\xi+\int_0^t T(t-s)F(C_\Lambda X(s))ds+\int_0^tT(t-s)\mathscr{M}(X(s))dW(s).
\end{gather*}
\end{example}

\section{The Feller property of the transition semigroup}\label{sec:feller}

We are interested here in the continuous dependence of the mild solution of (\ref{0}) on the initial data and the boundedness of the moments. Specifically, we have the following.

\begin{theorem}\label{prop2}
Under the assumptions of Theorem \ref{thm1}, there exists $C_T > 0$ such that, for arbitrary $\xi$, $\eta\in L^2(\mathcal{F}_0,H)$ and $t\in[0,T]$, the following estimate holds
$$\mathbb{E}\|X(t,\xi)-X(t,\nu)\|^2\leq C_T\mathbb{E}\|\xi-\nu\|^2.$$
\end{theorem}

\begin{proof}
Let $\xi$, $\eta\in L^2(\mathcal{F}_0,H)$ and set $M=\sup_{t\in[0,T]}\|T(t)\|$. Let $T_0>0$ such that $3k^2c(T_0)^2<\frac{1}{2}$. By the uniqueness of the solution we have that
$$X(t,\xi)=\mathbbm{1}_{[0,T_0]}X_1(t,\xi)+\mathbbm{1}_{[T_0,2T_0]} X_2(t,\xi), \quad\forall t\in[0,2T_0], $$
where,
\begin{align*}X_1(t,\xi)=T(t)\xi+\int_{0}^tT(t-s)&F(C_{\Lambda}X_1(\xi,s))ds\cr &+\int_0^tT(t-s)\mathscr{M}(X_1(s,\xi))dW(s)\end{align*}
and
\begin{align*}X_2(t,\xi)=T(t-T_0)X_1(T_0,\xi)&+\int_{T_0}^tT(t-s)F(C_\Lambda X_2(s,\xi))ds\cr & +\int_{T_0}^tT(t-s)\mathscr{M}\left(X_2(s,\xi)\right)dW(s).
\end{align*}
For all $t\in[0,T_0]$ we have
\begin{align*}
	& \|C_{\Lambda}X_1(s,\xi)-C_{\Lambda}X_1(s,\eta)\|^2\cr  &\qquad\leq  3\|C_{\Lambda}T(s)(\xi-\eta)\|^2\\
	& \qquad\qquad+3\left\|C_{\Lambda}\int_0^sT(s-\sigma)\left[F(C_{\Lambda}X_1(\sigma,\xi))-F(C_{\Lambda}X_1(\sigma,\eta))\right]d\sigma\right\|^2\\
	& \qquad\qquad+3 \left\|C_{\Lambda}\int_0^sT(s-\sigma)\left[\mathscr{M}(X_1(\sigma,\xi))-\mathscr{M}(X_1(\sigma,\eta))\right]dW(\sigma)\right\|^2.
\end{align*}
Using Proposition \ref{5*} and Proposition \ref{6*} we get
\begin{align*}
	& \int_0^t\mathbb{E}\|C_{\Lambda}X_1(s,\xi)-C_{\Lambda}X_1(s,\eta)\|^2ds\cr
	& \qquad \leq 3\gamma(T_0)^2\mathbb{E}\|\xi-\eta\|^2\\
	& \qquad\qquad +3k^2c(T_0)^2\int_0^t\mathbb{E}\left\|C_{\Lambda}X_1(s,\xi)-C_{\Lambda}X_1(s,\eta)\right\|^2ds\\
	& \qquad \qquad +3k^2\gamma(T_0)^2\int_0^t\mathbb{E}\left\|X_1(s,\xi)-X_1(s,\eta)\right\|^2ds.
\end{align*}
Then
\begin{align}\label{2}
	\begin{split}
		& \int_0^t\mathbb{E}\|C_{\Lambda}X_1(s,\xi)-C_{\Lambda}X_1(s,\eta)\|^2ds \cr &\qquad\leq 6\gamma(T_0)^2\mathbb{E}\|\xi-\eta\|^2+
		6k^2\gamma(T_0)^2\int_0^t\mathbb{E}\left\|X_1(s,\xi)-X_1(s,\eta)\right\|^2ds,
	\end{split}
\end{align}
for all $t\in[0,T_0]$. Repeating the same argument as above and using a change of variables, we obtain
\begin{align*}
	&\int_{T_0}^t\mathbb{E}\|C_{\Lambda}X_2(s,\xi)-C_{\Lambda}X_2(s,\eta)\|^2ds\cr  & \qquad\qquad\leq 6\gamma(T_0)^2\mathbb{E}\|X(T_0,\xi)-X(T_0,\eta)\|^2\cr & \qquad\qquad\qquad+6k^2\gamma(T_0)^2\int_{T_0}^t\mathbb{E}\left\|X_2(s,\xi)-X_2(s,\eta)\right\|^2ds,
\end{align*}
for all $t\in [T_0,2T_0]$. Moreover, simple calculations lead to
\begin{align*}
	& \mathbb{E}\|X(T_0,\xi)-X(T_0;\eta)\|^2\cr &\qquad\leq 3M^2\mathbb{E}\|\xi-\mu\|^2\\
	&\qquad\qquad+3T_0M^2k^2\mathbb{E}\int_0^{T_0}\|C_{\Lambda}X_1(s,\xi)-C_{\Lambda}X_1(s,\eta)\|^2ds\\
	& \qquad\qquad+3M^2k^2\mathbb{E}\int_0^{T_0}\|X_1(s,\xi)-X_1(s,\eta)\|^2ds.
\end{align*}
Using \eqref{2}, we get
\begin{align*}
	& \mathbb{E}\|X(T_0,\xi)-X(T_0;\eta)\|^2\cr &\qquad \leq \left(3M^2+18T_0M^2k^2\gamma(T_0)^2\right)\mathbb{E}\|\xi-\mu\|^2\\
	& \qquad\qquad +\left(3M^2k^2+18M^2k^2T_0\gamma(T_0)^2\right)\mathbb{E}\int_0^{T_0}\|X_1(s,\xi)-X_1(s,\eta)\|^2ds.
\end{align*}
It follows that
\begin{align}\label{x2}
	\begin{split}
		& \int_{T_0}^t\mathbb{E}\|C_{\Lambda}X_2(s,\xi)-C_{\Lambda}X_2(s,\eta)\|^2ds\cr &\qquad\leq 6\gamma(T_0)^2\left(3M^2+18M^2k^2T_0\gamma(T_0)^2\right)\mathbb{E}\|\xi-\eta\|^2\\
		& \qquad+6\gamma(T_0)^2\left(3M^2k^2+18M^2k^2T_0\gamma(T_0)^2\right)\int_{0}^{T_0}\mathbb{E}\left\|X_1(s,\xi)-X_1(s,\eta)\right\|^2ds\\
		& \qquad+6k^2\gamma(T_0)^2\int_{T_0}^t\mathbb{E}\left\|X_2(s,\xi)-X_2(s,\eta)\right\|^2ds.
	\end{split}
\end{align}
Thus, from \eqref{2} and \eqref{x2}, we have
\begin{align}\label{x}
	\begin{split}
		&\int_0^t\mathbb{E}\|C_{\Lambda}X(s,\xi)-C_{\Lambda}X(s,\eta)\|^2ds\cr
		&\qquad\leq  C_{T_0,1}\mathbb{E}\|\xi-\eta\|^2+C_{T_0,2}\int_0^{t}\mathbb{E}\left\|X(s,\xi)-X(s,\eta)\right\|^2ds
	\end{split}
\end{align}
for all $t\in[0,2T_0]$, where $C_{T_0,1}$, $C_{T_0,2} $ are constants. This process can be repeated to obtain \eqref{x} for all $t\in[0,T]$, $T>0$. On the other hand, we have
\begin{eqnarray*}
	\mathbb{E}\|X(t,\xi)-X(t,\eta)\|^2 &\leq& 3M^2\mathbb{E}\|\xi-\eta\|^2\\
	& &+ 3M^2k^2T\int_0^t\mathbb{E}\|C_{\Lambda}X(s,\xi)-C_{\Lambda}X(s,\eta)\|^2ds\\
	& &+3M^2k^2\int_0^t\mathbb{E}\|X(s,\xi)-X(s,\eta)\|^2ds,
\end{eqnarray*}
for all $t\in[0,T]$. Now, by using the inequality \eqref{x} we obtain
\begin{eqnarray*}
	\mathbb{E}\|X(t,\xi)-X(t,\eta)\|^2 &\leq& C_{T,3}\mathbb{E}\|\xi-\eta\|^2\\
	& &+C_{T,4}\int_0^t\mathbb{E}\|X(s,\xi)-X(s,\eta)\|^2ds,
\end{eqnarray*}
where $C_{T,3}$, $C_{T,4} $ are constants. Using Gronwall’s lemma, the estimate of the theorem follows.
\end{proof}

With the use of Theorem \ref{thm1} and Theorem \ref{prop2}, we can now prove the Feller property of the associated semigroup associated to the stochastic evolution equation \eqref{0}. In fact, under the condition of Theorem \ref{thm1}, we define
\begin{align*}P_t\phi(x):=\mathbb{E}(\phi(X(t,x))),\quad t\in[0,T],\; x\in H,\end{align*} where $\phi(\cdot)$ is bounded and Borel measurable function. In addition, we select
\begin{align*}
\mathcal{C}_b(H):=\left\{\phi:H\to\mathbb{R}: \phi\quad \text{is bounded and continuous function}\right\}.
\end{align*}
We have the following result.
\begin{proposition}
Under the hypotheses \textbf{(S)}, \textbf{(A)} and \textbf{(L)} we have  \begin{align*}P_t\mathcal{C}_b(H)\subset \mathcal{C}_b(H)\end{align*} for all $t\in [0,T]$.
\end{proposition}

\section{Applications to semilinear stochastic equations of neutral type}\label{sec:neutral}
We consider the following semi-linear stochastic equation of neutral type
\begin{align}\label{app}
\begin{cases}
d(X(t)-\mathscr{D}X_t)=[A(X(t)-\mathscr{D}X_t)+F(LX_t)]dt\cr\hspace{5.25cm}+B(X_t)dW(t), &t\in[0,T]\\
\left(X(t)-\mathscr{D}X_t\right)_{|_{t=0}}=\xi, \qquad X_0=\phi,
\end{cases}
\end{align}
where $A:D(A) \subset H\to H$ is the generator of a $C_0$-semigroup $(T(t))_{t\geq0}$ on a Hilbert space $H$, the process $X:[-r,\infty)\to X$ and its history $X_t:=X(t+\cdot):[-r,0]\to H$ is defined for any $t\ge 0$ by
\begin{align*}
X_t(\theta)=\begin{cases}X(t+\theta),& -t\le \theta\ge 0,\cr \phi(t+\theta),& -r\le \theta \le -t.\end{cases}
\end{align*}
The initial conditions $X(0)=\xi\in L^2(\mathcal{F}_0,H)$, $\phi\in L^2_{\mathcal{F}_0}([-r,0],H)$ and $W(t)$ is a real standard Wiener process. Here $F:H\to H$, $B:L^2([-r,0],H)\to H$ are nonlinear Lipschitz continuous and $L,\mathscr{D}:W^{1,2}([-r,0],H)\to H$ are the following Riemann-Stieltjes integrals
\begin{align}\label{RS-opera}L\phi:=\int_{-r}^0 d\mu(\theta)\phi(\theta),\quad\text{and}\quad \mathscr{D}\phi:=\int_{-r}^0 d\nu(\theta)\phi(\theta),\end{align}
where $\nu,\mu:[-r,0]\to \mathcal{L}(H)$  are of bounded variation and non-atomic at zero; that is
$$\lim_{\epsilon\to 0}|\mu|([-\epsilon,0])=0, \qquad \lim_{\epsilon\to 0}|\nu|([-\epsilon,0])=0.$$
\begin{definition}
	A pair $(X(t),X(t+\cdot))$ is a solution of \eqref{app} if there exist extensions $\tilde{\mathscr D}$ and $\tilde{L}$ of $\mathscr D$ and $L$, respectively, such that
	\begin{align*}
		&X_t\in D(\tilde{\mathscr D})\cap D(\tilde{L}) \quad \text{for a.e}~ t\geq0, \mathbb{P}- \text{almost surely},\;and\cr &
		X(t)=\tilde{\mathscr{D}} X_t +T(t)\xi+\int^t_0 T(t-s)F(\tilde{L}X_s)ds+\int^t_0 T(t-s)B(X_s)dW(s).
	\end{align*}
\end{definition}
The equation \eqref{app} was treated by Webb \cite[Section 4]{Web1976}  in the particular case $\mathscr{D}\equiv B\equiv 0$. The method employed consists in constructing a semigroup of nonlinear operators which may be associated with the solutions of this equation.

In order to reformulate the equation \eqref{app} to the equation \eqref{0}, we select
\begin{align}\label{difference-equation}
Z(t):=X(t)-\mathscr{D}X_t,\qquad t\ge 0.
\end{align}
Then the neutral equation \eqref{app} becomes boundary valued stochastic equation
\begin{align}\label{app1}
\begin{cases}
dZ(t)=[AZ(t)+F(LX_t)]dt+B(X_t)dW(t), &t\in[0,T]\\
Z(0)=\xi, \qquad X_0=\phi\\
X(t)=Z(t)+\mathscr{D}X_t,& t\ge 0.
\end{cases}
\end{align}
We now introduce product spaces
\begin{align*}
\mathcal{H}:=H\times L^2([-r,0],H),\qquad \mathcal{Z}:=D(A)\times W^{1,2}([-r,0],H).
\end{align*}
The space $\mathcal{H}$ is endowed with the following norm
\begin{align*}
\left\|(\begin{smallmatrix}x\\ g\end{smallmatrix}) \right\|:=\|x\|_H+\|g\|_{L^2([-r,0],H)}
\end{align*}
On the other hand, we introduce the operators
\begin{align*}
& \mathfrak{A}:=\begin{pmatrix} A& 0\\ 0&\frac{d}{d\theta}\end{pmatrix},\quad D(\mathfrak{A}):=\left\{ (\begin{smallmatrix}x\\ g\end{smallmatrix})\in \mathcal{Z}: \psi(0)=x+\mathscr{D}\psi \right\}\cr & \mathcal{P}:=\begin{pmatrix} 0& L\\ 0&0\end{pmatrix}:D(\mathfrak{A})\to \mathcal{H},\cr &
 Qg=g',\qquad D(Q)=\{g\in W^{1,2}([-r,0],H):g(0)=0\},\cr&
 \mathcal{F}: \mathcal{H}\to \mathcal{H},\quad \mathcal{F}(\begin{smallmatrix}x\\ g\end{smallmatrix})=(\begin{smallmatrix}F(x)\\ 0\end{smallmatrix}),\cr & \mathscr{M}:\mathcal{H}\to \mathcal{H},\quad \mathscr{M}(\begin{smallmatrix}x\\ g\end{smallmatrix})=(\begin{smallmatrix}B(g)\\ 0\end{smallmatrix}).
\end{align*}
Moreover, we introduce the following new state
\begin{align*}
\varrho(t)=\left(\begin{smallmatrix}Z(t)\\ X_t\end{smallmatrix}\right),\qquad t\ge 0.
\end{align*}
With the above notation, the problem \eqref{app1} is reformulated as
\begin{align}\label{app3}
\begin{cases}
d\varrho(t)=[\mathfrak{A}\varrho(t)+\mathcal{F}(\mathcal{P}\varrho(t))]dt+\mathscr{M}(\varrho(t))dW(t), &t\in[0,T]\\
\varrho(0)=(\begin{smallmatrix}\xi\\ \phi\end{smallmatrix}).
\end{cases}
\end{align}
\begin{theorem}
Let assumption \textbf{(S)} be satisfied. The neutral equation \eqref{app} admits a unique solution $(X(t),X(t+\cdot))$ such that
	\begin{align*}
	&X_t\in D(\mathscr{D}_{0,\Lambda})\cap D(L_{0,\Lambda}) \quad \text{for a.e}~ t\geq0, \mathbb{P}- \text{almost surely},\;and\cr &
X(t)=\mathscr{D}_{0,\Lambda} X_t+T(t)\xi+\int^t_0 T(t-s)F(L_{0,\Lambda}X_s)ds+\int^t_0 T(t-s)B(X_s)dW(s),
\end{align*}
where $L_{0,\Lambda}$ and $\mathscr{D}_{0,\Lambda}$ are the Yosida extensions of $L$ and $\mathscr{D}$ with respect to $Q$, respectively.
\end{theorem}
\begin{proof}
It is known that (see e.g. \cite{hadd-jfa}, \cite{hadd2008feedback}) that the operator $\mathfrak{A}$ is the generator of a $C_0$--semigroup $(\mathfrak{T}(t))_{t\ge 0}$ on $\mathcal{H}$. On the other hand, as $F:H\to H$ and $B:L^2([-r,0],H)\to H$ are Lipschitz functions, then it is so for the functions $\mathcal{F}$ and $\mathscr{M}$.  Hence the condition ${\bf (L)}$ is satisfied for the equation \eqref{app3}. In order to apply Theorem \ref{thm1}, it is necessary to show that the operators $\mathfrak{A}$ and $\mathcal{P}$ satisfy the condition ${\bf (A)}$.  It is known (see e.g. \cite[Chap.II]{engel2001one}) that $Q$ generates the left shift semigroup $(S(t))_{t\ge 0}$ on $L^2([-r,0],H)$ given by
\begin{align*}
	(S(t)g)(\theta):=
	\begin{cases}
		g(t+\theta), & t+\theta\leq 0,\\
		0, & t+\theta>0,
	\end{cases},
\end{align*}
for  $t\ge 0,$ $\theta\in [-r,0]$, and $g\in L^2([-r,0],H)$. On the other hand, we define the operators
\begin{align*}
	& \mathcal{A}:=\begin{pmatrix}A& 0\\0& Q\end{pmatrix},\qquad D(\mathcal{A})=D(A)\times D(Q),\cr
	& \mathcal{C}:=\left( I\;\;\mathscr{D}\right): D(\mathcal{A})\to H.
\end{align*}
Clearly, $\mathcal{A}$ generates a diagonal strongly continuous semigroup $(\mathcal{T}(t))_{t\ge 0}$  on  $\mathcal{H}$. As $\mathscr{D}_0=\mathscr{D}$ with domain $D(\mathscr{D}_0)=D(Q)$ is admissible for $Q,$ it follows that $\mathcal{C}$ is admissible for $\mathcal{A}$. We denote by $\mathcal{C}_\Lambda$ the Yosida extension of $\mathcal{C}$ with respect to $\mathcal{A}$. Moreover, it is proved in \cite[Theorem 19]{hadd2008feedback}, that ${\rm Im}(\mathfrak{T}(t))\subset D(\mathcal{C}_\Lambda)$ for almost every $t\ge 0,$ and for any $\alpha>0$ there exists a constant $\kappa(\alpha)>0$ such that
\begin{align}\label{estim1}
	\left\|\mathcal{C}_\Lambda\mathfrak{T}(\cdot)(\begin{smallmatrix}x\\ g\end{smallmatrix})\right\|_{L^2([0,\alpha],X)}\le \kappa(\alpha)\left\|(\begin{smallmatrix}x\\ g\end{smallmatrix})\right\|,\qquad \forall (\begin{smallmatrix}x\\ g\end{smallmatrix})\in \mathcal{H}.
\end{align}
Clearly we have
\begin{align}\label{calC-Lambda}
	\mathcal{C}_\Lambda= \left( I\;\;\mathscr{D}_{0,\Lambda}\right),\quad D(\mathcal{C}_\Lambda)=H\times D(\mathscr{D}_{0,\Lambda}),
\end{align}
where $\mathscr{D}_{0,\Lambda}$ is the Yosida extension of $\mathscr{D}$ with respect to $Q$.

In order to give an expression of the semigroup $\mathfrak{T}$ which can help us to prove the admissibility of $\mathcal{P}$ for $\mathfrak{A},$ we select
\begin{align*}
	\left(\Upsilon_t u\right)(\theta)=\begin{cases}u(t+\theta),& t+\theta\ge 0,\cr 0,& t+\theta<0\end{cases},
\end{align*}
for any $t\ge 0,\;\theta\in [-r,0]$ and $u\in L^2(\mathbb{R}^+,H)$.
By using
\cite[Theorem 19]{hadd2008feedback}, the semigroup $\mathfrak{T}$ satisfies
\begin{align}\label{sg-frakT}
	\mathfrak{T}(t)(\begin{smallmatrix}x\\ g\end{smallmatrix})=\begin{pmatrix}T(t)x\\ S(t)g+\Upsilon_t \mathcal{C}_\Lambda\mathfrak{T}(\cdot)(\begin{smallmatrix}x\\ g\end{smallmatrix})\end{pmatrix}
\end{align}
for any $t\ge 0$ and $(\begin{smallmatrix}x\\ g\end{smallmatrix})\in\mathcal{H}$. We select,
\begin{align*}
	L_0:=L_{|D(Q)}.
\end{align*}
Denote by $L_{0,\Lambda}$ the Yosida extension of $L_0$ with respect to $Q$.  According to \cite[Theorem 3]{HIR-2006}, $L_0$ is an admissible operator for $Q,$ ${\rm Im}(\Upsilon_t)\subset D(L_{0,\Lambda})$ and
\begin{align}\label{estim2}
	\|L_{0,\Lambda}\Upsilon_t u\|_{L^2([0,\alpha],H)}\le c_\alpha\|u\|_{L^2([0,t],H)},
\end{align}
for any $\alpha>0,$ $u\in L^2_{loc}(\mathbb{R}^+,H)$ and some constant $c_\alpha>0$. On the other hand, by using \cite[Lemma 3.6]{HMR-2015}, we have $W^{1,2}([-r,0],H)\subset D(L_{0,\Lambda})$ and $L_{0,\Lambda}\equiv L$ on $W^{1,2}([-r,0],H)$. Thus, for any $(\begin{smallmatrix}x\\ g\end{smallmatrix})\in D(\mathfrak{A}),$
\begin{align*}
	\mathcal{P}\mathfrak{T}(t)(\begin{smallmatrix}x\\ g\end{smallmatrix})&=\begin{pmatrix}0&L_{0,\Lambda}\\0&0\end{pmatrix}\mathfrak{T}(t)(\begin{smallmatrix}x\\ g\end{smallmatrix})\cr
	&=\begin{pmatrix}L_{0,\Lambda}S(t)g+L_{0,\Lambda}\Upsilon_t \mathcal{C}_\Lambda\mathfrak{T}(\cdot)(\begin{smallmatrix}x\\ g\end{smallmatrix})\\ 0\end{pmatrix}.
\end{align*}
Thus the admissibility of $\mathcal{P}$ for $\mathfrak{A}$ now follows from the admissibility of $L_0$ for $Q$, and the inequalities \eqref{estim1} and \eqref{estim2}. Thus the condition {\bf (A)} holds as well. Now by Theorem \ref{thm1}, the semi-linear stochastic equation \eqref{app3} (hence the neutral stochastic equation \eqref{app}) has a unique mild solution $\varrho\in \mathcal{C}_\mathcal{F}([0,T];\mathcal{H})$ such that
\begin{align}\label{VCF-varrho}
	\begin{split}
		\varrho(t)=\mathfrak{T}(t)(\begin{smallmatrix}\xi\\ \phi\end{smallmatrix})+\int^t_0 \mathfrak{T}(t-s)&\mathcal{F}(\mathcal{P}_{\Lambda,\mathfrak{A}}\varrho(s))ds\cr &+\int^t_0 \mathfrak{T}(t-s)\mathscr{M}(\varrho(s))dW(s)
	\end{split}
\end{align}
for any $t\ge 0$, $\xi\in L^2(\mathcal{F}_0,H)$ and $\phi\in L^2_{\mathcal{F}_0}([-r,0],H)$.

Let us now give an explicit expression to the first component of $\varrho(t)$. In fact, from \eqref{estim1} we $\mathcal{C}_\Lambda$ is an admissible operator for $\mathfrak{A}$. Thus by using \eqref{VCF-varrho}, Proposition \ref{5*} and Proposition \ref{6*}, we obtain
\begin{align*}
	\left(\begin{smallmatrix}Z(t)\\ X_t\end{smallmatrix}\right)=\varrho(t)\in D(\mathcal{C}_\Lambda)
\end{align*}
for almost every $t\ge 0$ and $\mathbb{P}$ almost surely. On the other hand, the relation \eqref{difference-equation} is a feedback low associated with the system \eqref{app} and can be reformulated as
\begin{align}\label{extended-difference}X(t)=\mathcal{C}_\Lambda \varrho(t)=Z(t)+\mathscr{D}_{0,\Lambda} X_t\end{align}for a.e. $t\ge 0$. Moreover, by using similar arguments as in \cite{hadd2008feedback}, one can see that $\Xi:=H\times [D(L_{0,\Lambda})\cap D(\mathscr{D}_{0,\Lambda}]\subset D(\mathcal{P}_{\Lambda,\mathfrak{A}})$ and
\begin{align}\label{P-lll}
	\mathcal{P}_{\Lambda,\mathfrak{A}}=\begin{pmatrix} 0& L_{0,\Lambda}\\ 0& 0\end{pmatrix}\quad\text{on}\quad \Xi.
\end{align}
Now by combining  \eqref{sg-frakT}, \eqref{extended-difference}, \eqref{P-lll} and \eqref{VCF-varrho}, the first component of $\varrho(t)$ is given by
\begin{align*}
	Z(t)&=X(t)-\mathscr{D}_{0,\Lambda} X_t\cr &=T(t)\xi+\int^t_0 T(t-s)F(L_{0,\Lambda}X_s)ds+\int^t_0 T(t-s)B(X_s)dW(s).
\end{align*}
\end{proof}
\begin{remark}
We mention that the case $\mathscr{D}\equiv 0$ corresponds to the standard delay equations.
\end{remark}

\end{document}